\documentclass[a4paper, oneside]{article}
\usepackage{enumitem}
\usepackage{lmodern}
\usepackage{amssymb,amsthm,amsmath}
\usepackage[utf8]{inputenc}
\usepackage{verbatim}
\usepackage{tikz}
\usepackage{authblk}
\usepackage{cite}
\usetikzlibrary{arrows}
\usetikzlibrary{decorations.pathreplacing}
\usetikzlibrary{automata}
\usetikzlibrary{positioning}

\theoremstyle{definition}
\newtheorem{definition}{Definition}[section]
\newtheorem{lemma}[definition]{Lemma}
\newtheorem{proposition}[definition]{Proposition}
\newtheorem{theorem}[definition]{Theorem}
\newtheorem{example}[definition]{Example}

\newtheorem{problem}[definition]{Problem}

\newcommand{\N}{\mathbb{N}}
\newcommand{\Npos}{{\mathbb{N}_+}}

\newcommand{\R}{\mathbb{R}}
\newcommand{\Z}{\mathbb{Z}}
\newcommand{\C}{\mathbb{C}}
\newcommand{\I}{\mathbb{I}}
\newcommand{\abs}[1]{{\left\vert #1 \right\vert}} % Itseisarvo
\newcommand{\digs}{\Sigma} % Symbol set consisting of digits
\newcommand{\fr}[1]{\mathrm{Frac}\left(#1\right)} % Fractional part

\DeclareMathOperator{\aut}{Aut} % Automorphism
\DeclareMathOperator{\End}{End} % Endomorphism
\DeclareMathOperator{\id}{Id} % Identity map
\DeclareMathOperator{\tr}{Tr} % Trace
\DeclareMathOperator{\cyl}{Cyl} % Cylinder set

\begin{document}

\title{Dynamics of cellular automata on beta-shifts and direct topological factorizations\thanks{The work was partially supported by the Academy of Finland grant 296018.}}

\author{Johan Kopra}

\affil{Department of Mathematics and Statistics, \\FI-20014 University of Turku, Finland}
\affil{jtjkop@utu.fi}

\date{}

\maketitle

\setcounter{page}{1}

\begin{abstract}
We consider the range of possible dynamics of cellular automata (CA) on two-sided beta-shifts $S_\beta$. We show that any reversible CA $F:S_\beta\to S_\beta$ has an almost equicontinuous direction whenever $S_\beta$ is not sofic. This has the corollary that non-sofic beta-shifts are topologically direct prime. We also make some preliminary observations on direct topological factorizations of beta-shifts that are subshifts of finite type. 
\end{abstract}

\providecommand{\keywords}[1]{\textbf{Keywords:} #1}
\noindent\keywords{cellular automata, beta-shifts, sensitivity, direct topological factorizations}

\section{Introduction}

Let $X\subseteq A^\Z$ be a one-dimensional subshift over a symbol set $A$. A cellular automaton (CA) is a function $F:X\to X$ defined by a local rule, and it endows the space $X$ with translation invariant dynamics given by local interactions. It is natural to ask how the structure of the underlying subshift $X$ affects the range of possible topological dynamics that can be achieved by CA on $X$. Our preferred approach is via the framework of directional dynamics of Sablik~\cite{Sab08}.

Using Sablik's terminology, it has been proven in Theorem~5.2.19 of~\cite{KopDiss} that every nontrivial mixing sofic subshift admits a reversible CA which is sensitive in all directions. On the other hand, Subsection~5.4.2 of~\cite{KopDiss} presents a collection of $S$-gap shifts $X_S$, all of them synchronizing and many with specification property, such that every reversible CA on $X_S$ has an almost equicontinuous direction. It would be interesting to extend the latter result to other natural classes of subshifts.

In this paper we consider two-sided beta-shifts, which form a naturally occurring class of mixing coded subshifts. We show that if $S_\beta$ is a non-sofic beta-shift, then every reversible CA on $S_\beta$ has an almost equicontinuous direction. As an application we use this result to show that non-sofic beta shifts are topologically direct prime, thus answering a problem suggested in the presentation~\cite{Mey13}. We then suggest a program of studying direct topological factorizations of sofic beta-shifts and accompany this suggestion with some preliminary remarks.

\section{Preliminaries}

In this section we recall some preliminaries concerning symbolic dynamics and topological dynamics in general. Good references to these topics are~\cite{Kur03,LM95}.

\begin{definition}If $X$ is a compact metrizable topological space and $T:X\to X$ is a continuous map, we say that $(X,T)$ is a \emph{(topological) dynamical system}.\end{definition}

When there is no risk of confusion, we may identify the dynamical system $(X,T)$ with the underlying space or the underlying map, so we may say that $X$ is a dynamical system or that $T$ is a dynamical system.

\begin{definition}We write $\psi:(X,T)\to (Y,S)$ whenever $(X,T)$ and $(Y,S)$ are dynamical systems and $\psi:X\to Y$ is a continuous map such that $\psi\circ T=S\circ\psi$ (this equality is known as the \emph{equivariance condition}). Then we say that $\psi$ is a \emph{morphism}. If $\psi$ is injective, we say that $\psi$ is an \emph{embedding}. If $\psi$ is surjective, we say that $\psi$ is a \emph{factor map} and that $(Y,S)$ is a factor of $(X,T)$ (via the map $\psi$). If $\psi$ is bijective, we say that $\psi$ is a \emph{conjugacy} and that $(X,T)$ and $(Y,S)$ are \emph{conjugate} (via $\psi$).\end{definition}

A finite set $A$ containing at least two elements (\emph{letters}) is called an \emph{alphabet}. In this paper the alphabet usually consists of numbers and thus for $n\in\Npos$ we denote $\digs_n=\{0,1,\dots,n-1\}$. The set $A^\Z$ of bi-infinite sequences (\emph{configurations}) over $A$ is called a \emph{full shift}. The set $A^\N$ is the set of one-way infinite sequences over $A$. Formally any $x\in A^\Z$ (resp. $x\in A^\N$) is a function $\Z\to A$ (resp. $\N\to A$) and the value of $x$ at $i\in\Z$ is denoted by $x[i]$. It contains finite, right-infinite and left-infinite subsequences denoted by $x[i,j]=x[i]x[i+1]\cdots x[j]$, $x[i,\infty]=x[i]x[i+1]\cdots$ and $x[-\infty,i]=\cdots x[i-1]x[i]$. Occasionally we signify the symbol at position zero in a configuration $x$ by a dot as follows:
\[x=\cdots x[-2]x[-1].x[0]x[1]x[2]\cdots.\]

A configuration $x\in A^\Z$ or $x\in A^\N$ is \emph{periodic} if there is a $p\in\Npos$ such that $x[i+p]=x[i]$ for all $i\in\Z$. Then we may also say that $x$ is $p$-periodic or that $x$ has period $p$. We say that $x$ is \emph{eventually periodic} if there is a $p\in\Npos$ such that $x[i+p]=x[i]$ holds for all sufficiently large $i\in\Z$.

A \emph{subword} of $x\in A^\Z$ is any finite sequence $x[i,j]$ where $i,j\in\Z$, and we interpret the sequence to be empty if $j<i$. Any finite sequence $w=w[1] w[2]\cdots w[n]$ (also the empty sequence, which is denoted by $\epsilon$) where $w[i]\in A$ is a \emph{word} over $A$. Unless we consider a word $w$ as a subword of some configuration, we start indexing the symbols of $w$ from $1$ as we have done here. The concatenation of a word or a left-infinite sequence $u$ with a word or a right-infinite sequence $v$ is denoted by $uv$. A word $u$ is a \emph{prefix} of a word or a right-infinite sequence $x$ if there is a word or a right-infinite sequence $v$ such that $x=uv$. Similarly, $u$ is a \emph{suffix} of a word or a left-infinite sequence $x$ if there is a word or a left-infinite sequence $v$ such that $x=vu$. The set of all words over $A$ is denoted by $A^*$, and the set of non-empty words is $A^+=A^*\setminus\{\epsilon\}$. The set of words of length $n$ is denoted by $A^n$. For a word $w\in A^*$, $\abs{w}$ denotes its length, i.e. $\abs{w}=n\iff w\in A^n$. For any word $w\in A^+$ we denote by ${}^\infty w$ and $w^\infty$ the left- and right-infinite sequences obtained by infinite repetitions of the word $w$. We denote by $w^\Z\in A^\Z$ the configuration defined by $w^\Z[in,(i+1)n-1]=w$ (where $n=\abs{w}$) for every $i\in\Z$.

Any collection of words $L\subseteq A^\Z$ is called a \emph{language}. For any $S\subseteq A^\Z$ the collection of words appearing as subwords of elements of $S$ is the language of $S$, denoted by $L(S)$. For $n\in\N$ we denote $L^n(S)=L(S)\cap A^n$. For any $L\subseteq A^*$, let
\[L^*=\{w_1\cdots w_n\mid n\geq 0,w_i\in L\}\subseteq A^*,\]
i.e. $L^*$ is the set of all finite concatenations of elements of $L$. If $\epsilon\notin L$, define $L^+=L^*\setminus\{\epsilon\}$ and if $\epsilon\in L$, define $L^+=L^*$.

To consider topological dynamics on subsets of the full shifts, the sets $A^\Z$ and $A^\N$ are endowed with the product topology (with respect to the discrete topology on $A$). These are compact metrizable spaces. The shift map $\sigma:A^\Z\to A^\Z$ is defined by $\sigma(x)[i]=x[i+1]$ for $x\in A^\Z$, $i\in\Z$, and it is a homeomorphism. Also in the one-sided case we define $\sigma:A^\N\to A^\N$ by $\sigma(x)[i]=x[i+1]$. Any topologically closed \emph{nonempty} subset $X\subseteq A^\Z$ such that $\sigma(X)=X$ is called a \emph{subshift}. A subshift $X$ equipped with the map $\sigma$ forms a dynamical system and the elements of $X$ can also be called \emph{points}. Any $w\in L(X)\setminus{\epsilon}$ and $i\in\Z$ determine a \emph{cylinder} of $X$
\[\cyl_X(w,i)=\{x\in X\mid w \mbox{ occurs in }x\mbox{ at position }i\}.\]

\begin{definition}
We say that a subshift $X$ is \emph{transitive} if for all words $u,v\in L(X)$ there is $w\in L(X)$ such that $uwv\in L(X)$. We say that $X$ is \emph{mixing} if for all $u,v\in L(X)$ there is $N\in\N$ such that for all $n\geq N$ there is $w\in L^n(X)$ such that $uwv\in L(X)$.
\end{definition}

\begin{definition}
Let $X\subseteq A^\Z$ and $Y\subseteq B^\Z$ be subshifts. We say that the map $F:X\to Y$ is a \emph{sliding block code} from $X$ to $Y$ if there exist integers $m\leq a$ (memory and anticipation) and a \emph{local rule} $f:A^{a-m+1}\to B$ such that $F(x)[i]=f(x[i+m],\dots,x[i],\dots,x[i+a])$. The quantity $d=a-m$ is the \emph{diameter} of the local rule $f$. If $X=Y$, we say that $F$ is a \emph{cellular automaton} (CA). If we can choose $f$ so that $-m=a=r\geq 0$, we say that $F$ is a radius-$r$ CA.\end{definition}

The following observation of~\cite{Hed69} characterizes sliding block codes as the class of structure preserving transformations between subshifts.

\begin{theorem}[Curtis-Hedlund-Lyndon]
A map $F:X\to Y$ between subshifts $X$ and $Y$ is a morphism between dynamical systems $(X,\sigma)$ and $(Y,\sigma)$ if and only if it is a sliding block code.
\end{theorem}

Bijective CA are called \emph{reversible}. We denote by $\End(X)$ the monoid of CA on $X$ and by $\aut(X)$ the group of reversible CA on $X$ (the binary operation is function composition). Hedlund's theorem allows us to construct cellular automata $F:X\to X$ without explicitly giving any local rule: it is sufficient to define $F$ so that it is continuous and that it commutes with $\sigma:X\to X$.

The notions of almost equicontinuity and sensitivity can be defined for general topological dynamical systems. We omit the topological definitions, because for cellular automata on transitive subshifts there are combinatorial characterizations for these notions using blocking words.

\begin{definition}
Let $F:X\to X$ be a radius-$r$ CA and $w\in L(X)$. We say that $w$ is a \emph{blocking word} if there is an integer $e$ with $\abs{w}\geq e\geq r+1$ and an integer $p\in[0,\abs{w}-e]$ such that
\[\forall x,y\in\cyl_X(w,0), \forall n\in\N, F^n(x)[p,p+e-1]=F^n(y)[p,p+e-1].\]
\end{definition}

The following is proved in Proposition~2.1 of~\cite{Sab08}.

\begin{proposition}\label{equiblock}
If $X$ is a transitive subshift and $F:X\to X$ is a CA, then $F$ is almost equicontinuous if and only if it has a blocking word.
\end{proposition}

We say that a CA on a transitive subshift is \emph{sensitive} if it is not almost equicontinuous. The notion of sensitivity is refined by Sablik's framework of directional dynamics~\cite{Sab08}.

\begin{definition}
Let $F:X\to X$ be a cellular automaton and let $p,q\in\Z$ be coprime integers, $q>0$. Then $p/q$ is a \emph{sensitive direction} of $F$ if $\sigma^p\circ F^q$ is sensitive. Similarly, $p/q$ is an \emph{almost equicontinuous direction} of $F$ if $\sigma^p\circ F^q$ is almost equicontinuous.
\end{definition}

This definition is best understood via the \emph{space-time diagram} of $x\in X$ with respect to $F$, in which successive iterations $F^t(x)$ are drawn on consecutive rows (see Figure~\ref{shift} for a typical space-time diagram of a configuration with respect to the shift map). By definition $-1=(-1)/1$ is an almost equicontinuous direction of $\sigma:A^\Z\to A^\Z$ because $\sigma^{-1}\circ \sigma=\id$ is almost equicontinuous. This is directly visible in the space-time diagram of Figure~\ref{shift}, because it looks like the space-time diagram of the identity map when it is followed along the dashed line. Note that the slope of the dashed line is equal to $-1$ with respect to the vertical axis extending downwards in the diagram.

\begin{figure}[ht]
\centering
\includegraphics{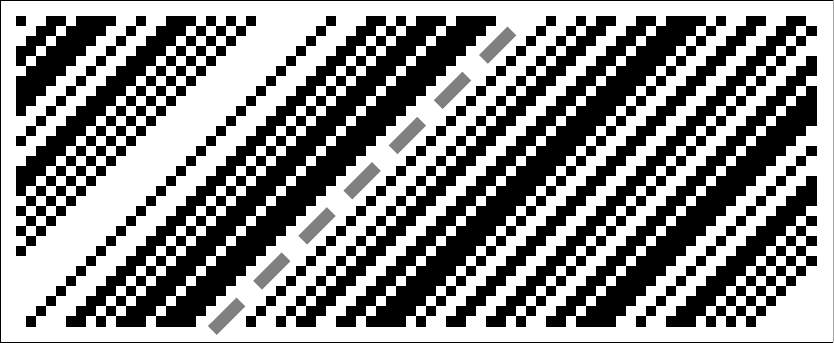}
\caption{A space-time diagram of the binary shift map $\sigma$. White and black squares correspond to digits $0$ and $1$ respectively. The dashed line shows an almost equicontinuous direction.}
\label{shift}
\end{figure}
%A direction of non-sensitivity is indicated by a red line.}

The notions of subshifts of finite type (SFT) and sofic subshifts are well known and can be found from Chapters~2 and~3 of~\cite{LM95}. Any square matrix $A$ with nonnegative entries is an adjacency matrix of a directed graph with multiple edges. The set of all bi-infinite sequences of edges forming valid paths is an \emph{edge SFT} (associated to $A$), whose alphabet is the set of edges.

The zeta-function $\zeta_X(t)$ of a subshift $X$ is a formal power series that encodes information about the number of periodic configurations in $X$ and it is a conjugacy invariant of $X$ (for precise definitions see Section~6.4 of~\cite{LM95}). Every SFT $X$ is conjugate to an edge SFT associated to a square matrix $A$. Let $I$ be an index set and let $\{\mu_i\in\C\setminus\{0\}\mid i\in I\}$ be the collection of non-zero eigenvalues of $A$ with multiplicities: it is called the non-zero spectrum of $X$. It is known that then $\zeta_X(t)=\prod_{i\in I}(1-\mu_i t)^{-1}$. The number of $p$-periodic configurations in $X$ is equal to $\sum_{i\in I}\mu_i^p$ for $p\in\Npos$. If $Y$ is also an SFT with $\zeta_Y(t)=\prod_{j\in J}(1-\nu_j t)^{-1}$, then the zeta-function of $X\times Y$ is $\zeta_X(t)\otimes\zeta_Y(t)=\prod_{i\in I, j\in J}(1-\mu_i\nu_j t)^{-1}$~\cite{Lind84}. 

Some other classes of subshifts relevant to the study of beta-shifts are the following.

\begin{definition}
Given a subshift $X$, we say that a word $w\in L(X)$ is synchronizing if
\[\forall u,v\in L(X): uw,wv\in L(X)\implies uwv\in L(X).\]
We say that a transitive subshift $X$ is synchronizing if $L(X)$ contains a synchronizing word.
\end{definition}

\begin{definition}
A language $L\subseteq A^+$ is a \emph{code} if for all distinct $u,v\in L$ it holds that $u$ is not a prefix of $v$. A subshift $X\subseteq A^\Z$ is a \emph{coded subshift} (given by a code $L$) if $L(X)$ is the set of all subwords of elements of $L^*$.
\end{definition}

\section{Beta-shifts}

We recall some preliminaries on beta-shifts from Blanchard's paper~\cite{Bla89} and from Lothaire's book~\cite{Lot02}.

For $\xi\in\R$ we denote $\fr{\xi}=\xi-\left\lfloor \xi\right\rfloor$, for example $\fr{1.2}=0.2$ and $\fr{1}=0$.

\begin{definition}
For every real number $\beta>1$ we define a dynamical system $(\I,T_\beta)$, where $\I=[0,1]$ and $T_\beta(\xi)=\fr{\beta \xi}$ for every $\xi\in\I$.
\end{definition}

\begin{definition}
The $\beta$-expansion of a number $\xi\in\I$ is the sequence $d(\xi,\beta)\in \digs_{\left\lfloor\beta\right\rfloor+1}^\N$ where $d(\xi,\beta)[i]=\left\lfloor\beta T^{i}(\xi)\right\rfloor$ for $i\in\N$.
\end{definition}

Denote $d(1,\beta)=d(\beta)$. By this convention $d(2)=2000\dots$ If $d(\beta)$ ends in infinitely many zeros, i.e. $d(\beta)=d_0\cdots d_m 0^\infty$ for $d_m\neq 0$, we say that $d(\beta)$ is finite, write $d(\beta)=d_0\cdots d_m$, and define $d^*(\beta)=(d_0\cdots (d_m-1))^\infty$. Otherwise we let $d^*(\beta)=d(\beta)$. Denote by $D_\beta$ the set of $\beta$-expansions of numbers from $[0,1)$. It is the set of all infinite concatenations of words from the prefix code
\[Y_\beta=\{d_0d_1\cdots d_{n-1} b\mid n\in\N,0\leq b<d_n\}\]
where $d(\beta)=d_0d_1d_2\dots$. For example, $Y_2=\{0,1\}$. Let $S_\beta$ be the coded subshift given by the code $Y_\beta$. Since $S_\beta$ is coded, it also has a natural representation by a deterministic automaton (not necessarily finite)~\cite{Bla86,ThuPre}. These representations allow us to make pumping arguments similar to those that occur in the study of sofic shifts and regular languages.

The subshift $S_\beta$ is mixing because the code $Y_\beta$ determining it always contains the word $0$ of length $1$. It is sofic if and only if $d(\beta)$ is eventually periodic and it is an SFT if and only if $d(\beta)$ is finite.

There is a natural lexicographical ordering on $\digs_n^\N$ which we denote by $<$ and $\leq$. Using this we can alternatively characterize $S_\beta$ as
\[S_\beta=\{x\in \digs_{\lfloor\beta\rfloor}^\Z\mid x[i,\infty]\leq d^*(\beta)\mbox{ for all }i\in\Z\}.\]
We call $S_\beta$ a \emph{beta-shift} (with base $\beta$). When $\beta>1$ is an integer, the equality $S_\beta=\digs_\beta^\Z$ holds.

\section{CA dynamics on beta-shifts}

In this section we study the topological dynamics of reversible CA on beta-shifts, and more precisely the possibility of them having no almost equicontinuous directions. By Theorem~5.2.19 of~\cite{KopDiss} every nontrivial mixing sofic subshift admits a reversible CA which is sensitive in all directions, and in particular this holds for mixing sofic beta-shifts. In this section we see that this result does not extend to the class of non-sofic beta-shifts.

We begin with a proposition showing that a surjective CA on a non-sofic beta-shift has to ``fix the expansion of one'' in some sense.

\begin{proposition}\label{dbsurjstable}
Let $\beta>1$ be such that $S_\beta$ is not sofic, let $F\in\End(S_\beta)$ be surjective, let $x\in S_\beta$ be such that $x[0,\infty]=d(\beta)$ and let $y\in F^{-1}(x)$. Then there is a unique $i\in\Z$ such that $y[i,\infty]=d(\beta)$. Moreover, $i$ does not depend on the choice of $x$ or $y$.
\end{proposition}
\begin{proof}
Let $r\in\N$ be such that $F$ is a radius-$r$ CA.

We first claim that $i$ does not depend on the choice of $x$ or $y$ when it exists. To see this, assume to the contrary that for $j\in\{1,2\}$ there exist $x_j\in S_\beta$ with $x_j[0,\infty]=d(\beta)$, $y_j\in F^{-1}(x_j)$ and $i_j\in\Z$ such that $i_1<i_2$ and $y_1[i_1,\infty]=d(\beta)=y_2[i_2,\infty]$. Then in particular for $M=\max\{i_2-i_1,i_2\}$ it holds that
$y_2[M,\infty]=y_2[M-i_2+i_2,\infty]=y_1[M-i_2+i_1,\infty]$ and
\begin{flalign*}
& d(\beta)[M-i_2+i_1+r,\infty]=x_1[M-i_2+i_1+r,\infty]=F(y_1)[M-i_2+i_1+r,\infty] \\
& = F(y_2)[M+r,\infty]=x_2[M+r,\infty]=d(\beta)[M+r,\infty].
\end{flalign*}
Then $d(\beta)$ would be eventually periodic, contradicting the assumption that $S_\beta$ is not sofic.

For the other claim, let us assume that for some $x$ and $y$ as in the assumption of the proposition there is no $i\in\Z$ such that $y[i,\infty]=d(\beta)$ . We claim that the sequence $y[-r,\infty]$ can be written as an infinite concatenation of elements of $Y_\beta$. This concatenation is found inductively. By our assumption $y[-r,\infty]<d(\beta)$, so $y[-r,\infty]$ has a prefix of the form $w_1=d_0d_1\cdots d_{n-1}b\in Y_\beta$ for some $n\in\N$, $b<d_n$. We can write $y[-r,\infty]=w_1x_1$ for some $x_1\in\digs_{\lfloor\beta\rfloor}^\N$. Because $x_1$ is a suffix of $y$, then again from our assumption it follows that $x_1<d(\beta)$ and we can find a $w_2\in Y_\beta$ which is a prefix of $x_1$. For all $i\in\Z$ we similarly we find $w_i\in Y_\beta$ such that $y[-r,\infty]=w_1w_2w_3\dots$.

Let $r_i$ be such that $y[-r,r_i]=w_1\cdots w_i$ for all $i\in\N$. Fix some $j,k\in\N$ such that $0\leq r_j<r_k$, $\abs{r_k-r_j}\geq 2r$ and $y[r_j-r,r_j+r]=y[r_k-r,r_k+r]$. Because $x$ is not eventually periodic, it follows that $x[r_j+1,\infty]\neq x[r_k+1,\infty]$. 

Assume first that $x[r_j+1,\infty]<x[r_k+1,\infty]$. Because $S_\beta$ is coded, there is a configuration $z\in S_{\beta}$ such that $z[-r,\infty]=w_1\cdots w_{j}w_{k+1}w_{k+2}\cdots$, i.e. this suffix can be found by removing the word $w_{j+1}\cdots w_k$ from the middle of $y[-r,\infty]$. Then $F(z)\in S_\beta$ but $F(z)[0,\infty]=x[0,r_j]x[r_k+1,\infty]>x[0,r_j]x[r_j+1,\infty]=d(\beta)$ contradicting one of the characterizations of $S_\beta$.

Assume then alternatively that $x[r_j+1,\infty]>x[r_k+1,\infty]$. Because $S_\beta$ is coded, there is a configuration $z\in S_{\beta}$ such that
\[z[-r,\infty]=w_1\cdots w_{j}(w_{j+1}\cdots w_k)(w_{j+1}\cdots w_k)w_{k+1}w_{k+2}\cdots,\]
i.e. this suffix can be found by repeating the occurrence of the word $w_{j+1}\cdots w_k$ in the middle of $y[-r,\infty]$. Then $F(z)\in S_\beta$ but
\begin{flalign*}
&F(z)[0,\infty]=x[0,r_j]x[r_j+1,r_k]x[r_j+1,r_k]x[r_k+1,\infty] \\
&=x[0,r_j]x[r_j+1,r_k]x[r_j+1,\infty] >x[0,r_j]x[r_j+1,r_k]x[r_k+1,\infty]=d(\beta)
\end{flalign*}
contradicting again the characterization of $S_\beta$.
\end{proof}

If $S_\beta$ is not sofic, $F\in\End(S_\beta)$ is surjective and $i\in\Z$ is as in the previous proposition, we say that the intrinsic shift of $F$ is equal to $i$. If the intrinsic shift of $F$ is equal to $0$, we say that $F$ is shiftless.

In the class of non-synchronizing beta-shifts we get a very strong result on surjective CA: they are all shift maps.

\begin{theorem}
If $S_\beta$ is not synchronizing, then all surjective CA in $\End(S_\beta)$ are powers of the shift map.
\end{theorem}
\begin{proof}
Let $F\in\End(S_\beta)$ be an arbitrary surjective CA and let $r\in\N$ be some radius of $F$. We may assume without loss of generality (by composing $F$ with a suitable power of the shift if necessary) that $F$ is shiftless. We prove that $F=\id$.

Assume to the contrary that $F\neq\id$, so there is $x\in S_\beta$ such that $F(x)[0]\neq x[0]$. Let $e={}^\infty 0.d(\beta)$ and let $z\in F^{-1}(e)$ be arbitary, so in particular $z[0,\infty]=d(\beta)$ by Proposition~\ref{dbsurjstable}. Since $S_\beta$ is not synchronizing, it follows that every word of $L(S_\beta)$ occurs in $d(\beta)$ (as explained by Kwietniak in~\cite{Kwi18}, attributed to Bertrand-Mathis~\cite{BerPre}). In particular it is possible to choose $i\geq r+1$ such that $\sigma^i(z)[-r,r]=x[-r,r]$ and $F(x)[0]=F(\sigma^i(z))[0]=\sigma^i(z)[0]=x[0]$, a contradiction.
\end{proof}

Next we consider only reversible CA. They do not have to be shift maps in the class of general non-sofic beta-shifts, and in fact the group $\aut(X)$ contains a copy of the free product of all finite groups whenever $X$ is an infinite synchronizing subshift by Theorem~2.17 of~\cite{Fie96}. Nevertheless $\aut(S_\beta)$ is constrained in the sense of directional dynamics.

\begin{theorem}\label{betaEqui}
If $S_\beta$ is not sofic and $F\in\aut(S_\beta)$ is shiftless then $F$ admits a blocking word. In particular all elements of $\aut(S_\beta)$ have an almost equicontinuous direction.
\end{theorem}
\begin{proof}
Let $r\in\Npos$ be a radius of both $F$ and $F^{-1}$. By the Morse-Hedlund theorem there is a word $u\in\digs_{\lfloor\beta\rfloor}^{3r}$ and symbols $a<b$ such that both $ua$ and $ub$ are subwords of $d(\beta)$. Let $p=p'ub$ ($p,p'\in L(S_\beta)$) be some prefix of $d(\beta)$ ending in $ub$. We claim that $p$ is blocking. More precisely we will show that if $x\in S_\beta$ is such that $x[0,\abs{p}-1]=p$ then $F^t(x)[0,\abs{p}-2]=p'u$ for all $t\in\N$.

Assume to the contrary that $t\in\N$ is the minimal number for which we have $F^t(x)[0,\abs{p}-2]\neq p'u$. We can find $w,v,v'\in L(S_\beta)$ and $c,d\in\digs_{\lfloor\beta\rfloor}$ ($c<d$) so that $u=wdv$, $\abs{w}\geq 2r$ and $F^t(x)[0,\abs{p}-2]=p'wcv'$. Indeed, $F^{-1}$ is shiftless because $F$ is, and therefore the prefix $p'w$ still remains unchanged in $F^t(x)[0,\infty]$.

Now we note that $x$ could have been chosen so that some of its suffixes is equal to $0^\infty$ and in particular under this choice no suffix of $F^t(x)$ is equal to $d(\beta)$. As in the proof of Proposition~\ref{dbsurjstable} we can represent $F^t(x)[0,\infty]=w_1w_2w_3\dots$ where $w_i\in Y_\beta$ for all $i\in\N$ and in fact $w_1=p'wc$.

Now let $q=q'ua$ ($q,q'\in L(S_\beta)$) be some prefix of $d(\beta)$ ending in $ua$. Then also $q'wd$ is a prefix of $d(\beta)$ and thus $q'wc\in Y_\beta$. Because $S_\beta$ is a coded subshift, there is a configuration $y\in S_\beta$ such that $y[0,\infty]=(q'wc)w_2w_3\dots$. For such $y$ it holds that $F^{-t}(y)\in S_\beta$ but $F^{-t}(y)[0,\infty]=q'(ub) x[\abs{p},\infty]>d(\beta)$ contradicting the characterization of $S_\beta$.
\end{proof}

\section{Topological direct primeness of beta-shifts}

We recall the terminology of Meyerovitch~\cite{Mey17}. A \emph{direct topological factorization} of a subshift $X$ is a subshift $X_1\times\cdots\times X_n$ which is conjugate to $X$ and where each $X_i$ is a subshift. We also say that each subshift $X_i$ is a \emph{direct factor} of $X$. The subshift $X$ is \emph{topologically direct prime} if it does not admit a non-trivial direct factorization, i.e. if every direct factorization contains one copy of $X$ and the other $X_i$ in the factorization contain just one point.

Non-sofic $\beta$-shifts turn out to be examples of topologically direct prime dynamical systems. This is an application of Theorem~\ref{betaEqui}.

\begin{theorem}\label{betaPrime}
If $S_\beta$ is a non-sofic beta-shift then it is topologically direct prime.
\end{theorem}
\begin{proof}
Assume to the contrary that there is a topological conjugacy $\phi:S_\beta\to X\times Y$ where $X$ and $Y$ are non-trivial direct factors of $S_\beta$. The subshifts $X$ and $Y$ are mixing as factors of the mixing subshift $S_\beta$, and in particular both of them are infinite, because a mixing finite subshift can only contain one point.

Define a reversible CA $F:X\times Y\to X\times Y$ by $F(x,y)=(\sigma(x),\sigma^{-1}(y))$ for all $x\in X$, $y\in Y$. Because $X$ and $Y$ are infinite, it follows that $F$ has no almost equicontinuous directions, i.e. $\sigma^r\circ F^s$ is sensitive for all coprime $r$ and $s$ such that $s>0$. Then define $G=\phi^{-1}\circ F\circ\phi: S_\beta\to S_\beta$. The map $G$ is a reversible CA on $S_\beta$ and furthermore $(S_\beta,G)$ and $(X\times Y,F)$ are conjugate via the map $\phi$. By Theorem~\ref{betaEqui} the CA $G$ has an almost equicontinuous direction, so we can fix coprime $r$ and $s$ such that $s>0$ for which $\sigma^r\circ G^s$ is almost equicontinuous. But $\sigma^r\circ G^s$ is conjugate to $\sigma^r\circ F^s$ via the map $\phi$, so $\sigma^r\circ F^s$ is also almost equicontinuous, a contradiction.
\end{proof}

In general determining whether a given subshift is topologically direct prime or not seems to be a difficult problem. Lind gives sufficient conditions in~\cite{Lind84} for SFTs based on their entropies: for example any mixing SFT with entropy $\log p$ for a prime number $p$ is topologically direct prime. The paper~\cite{Mey17} contains results on multidimensional full shifts, multidimensional 3-colored chessboard shifts and $p$-Dyck shifts with $p$ a prime number.

In the class of beta-shifts the question of topological direct primeness remains open in a countable number of cases.

\begin{problem}\label{primeProb}
Characterize the topologically direct prime sofic beta-shifts.
\end{problem}

\begin{example}
If $\beta>1$ is an integer, then $S_\beta=\digs_\beta^\Z$ is topologically direct prime if and only if $\beta$ is a prime number. Namely, if $\beta=nm$ for integers $n,m\geq 2$, then $S_\beta$ is easily seen to be conjugate to $S_n\times S_m$ via a coordinatewise symbol permutation. The case when $\beta=p$ is a prime number follows by Lind's result because the entropy of $\digs_p^\Z$ is $\log p$. 
\end{example}

In this example the existence of a direct factorization is characterized by the existence of direct factorization into beta shifts with integral base. Therefore, considering the following problem might be a good point to start with Problem~\ref{primeProb}.

\begin{problem}
Characterize the numbers $n,\gamma>1$ such that $n$ is an integer and $S_{n\gamma}$ is conjugate to $S_n\times S_{\gamma}$.
\end{problem}

In the following theorem we consider this simpler problem in the SFT case. We start with a definition and a lemma stated in~\cite{Bla89}.

\begin{definition}
Let $n>1$ be an integer, $a\in\digs_n$ and $w\in\digs_n^*$. We say that $aw$ is \emph{lexicographically greater than all its shifts} if $aw0^\infty>\sigma^i(aw0^\infty)$ for every $i>0$.
\end{definition}

\begin{lemma}\label{bSFTlemma}
$S_\beta$ is an SFT if and only if $\beta>1$ is the unique positive solution of some equation $x^d=a_{d-1}x^{d-1}+\dots +a_0$ where $d\geq 1$, $a_{d-1},a_0\geq 1$ and $a_i\in\N$ such that $a_{d-1}\cdots a_{0}$ is lexicographically greater than all its shifts. Then $d(\beta)=a_{d-1}\cdots a_0$.
\end{lemma}
\begin{proof}
The polynomial equation is equivalent to $1=a_{d-1}x^{-1}+\cdots +a_0 x^{-d}$, which clearly has a unique positive solution. If $\beta$ satisfies such an equation then $d(\beta)=a_{d-1}\cdots a_0$ and $S_\beta$ is an SFT. On the other hand, if $S_\beta$ is an SFT, then $d(\beta)$ takes the form of a word $a_{d-1}\cdots a_0$ which is lexicographically greater than all its shifts and $\beta$ satisfies $1=a_{d-1}x^{-1}+\cdots +a_0 x^{-d}$.
\end{proof}

\begin{theorem}\label{nSFTDecomp}
Let $S_\gamma$ be an SFT with $\gamma$ the unique positive solution of some equation $x^d=a_{d-1}x^{d-1}+\dots +a_0$ where $d\geq 1$, $a_{d-1},a_0\geq 1$ and $a_i\geq 0$ such that $a_{d-1}\cdots a_{0}$ is lexicographically greater than all its shifts. If $n\geq 2$ is an integer such that also $(na_{d-1})\cdots (n^d a_0)$ is lexicographically greater than all its shifts, then $S_{n\gamma}$ is topologically conjugate to $S_n\times S_\gamma$. The converse also holds: if $(na_{d-1})\cdots (n^d a_0)$ is not lexicographically greater than all its shifts, then either $S_{n\gamma}$ is not an SFT or $S_{n\gamma}$ and $S_n\times S_\gamma$ have different zeta functions. In particular they are not conjugate.
\end{theorem}
\begin{proof}
We have $d(\gamma)=a_{d-1}\cdots a_0$. The roots of $x^d=a_{d-1}x^{d-1}+\dots +a_0$ are $\gamma_1=\gamma,\gamma_2,\dots,\gamma_d$. By multiplying both sides by $n^d$ and by substituting $y=nx$ we see that the roots of $y^d=na_{d-1}y^{d-1}+\cdots+n^d a_0$ are $n\gamma_i$ and $n\gamma$ is the unique positive solution. Because multiplying $\prod_i(x-\gamma_i)=0$ by $n^d$ yields $\prod_i(y-n\gamma_i)=0$, we also see that the multiplicities of the roots $\gamma_i$ and $n\gamma_i$ are the same in their respective equations. If $(na_{d-1})\cdots (n^d a_0)$ is lexicographically greater than all its shifts, then $S_{n\gamma}$ is an SFT with $d(n\gamma)=na_{d-1}\cdots n^d a_0$. As in~\cite{ThuPre}, the shifts $S_\gamma$ and $S_{n\gamma}$ are conjugate to the edge shifts $X_C$ and $X_B$ respectively given by the matrices

\[C=\left(\begin{matrix} a_{d-1} & 1 & 0 & \cdots & 0 \\ a_{d-2} & 0 & 1 & \cdots & 0 \\ \vdots & \vdots & \vdots & & \vdots \\a_0 & 0 & 0 & \cdots & 0\end{matrix}\right) \qquad
B=\left(\begin{matrix} na_{d-1} & 1 & 0 & \cdots & 0 \\ n^2 a_{d-2} & 0 & 1 & \cdots & 0 \\ \vdots & \vdots & \vdots & & \vdots \\n^d a_0 & 0 & 0 & \cdots & 0\end{matrix}\right).
\]
They are also the companion matrices of the polynomials $x^d-a_{d-1}x^{d-1}-\dots-a_0$ and $y^d-na_{d-1}y^{d-1}-\cdots-n^d a_0$, so the eigenvalues are the roots of these polynomials and the zeta functions of $S_\gamma$ and $S_{n\gamma}$ are

\[\zeta_{X_C}(t)=\prod_i(1-\gamma_i t)^{-1} \quad \mbox{and} \quad \zeta_{X_B}(t)=\prod_i(1-n\gamma_i t)^{-1}.\]

In any case $\zeta_{S_n}=(1-nt)^{-1}$, so the zeta function of $X=S_n\times S_\gamma$ is $\zeta_X(t)=\prod_i(1-n\gamma_i t)^{-1}$, which is equal to $\zeta_{X_B}$. 

We will construct a conjugacy between $S_n\times X_C$ and $X_B$. We will choose the labels of the edges in $X_C$ and $X_B$ as in Figures~\ref{SFTconj1} and~\ref{SFTconj2}. The labels in the figures range according to $0\leq i_j<n$ and $0\leq k_j<a_{d-j}$ for $1\leq j\leq d$. 

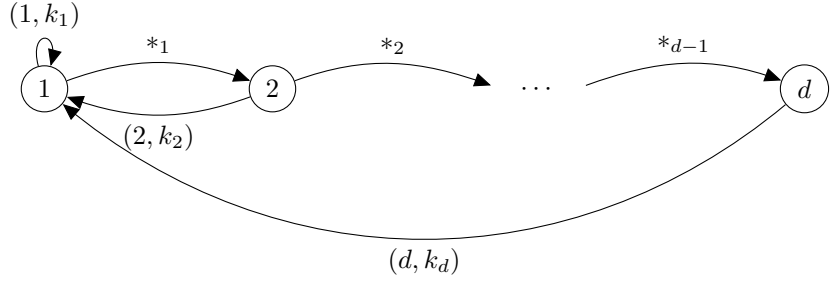
\begin{figure}
\centering
\begin{tikzpicture}[auto]
\node (d1) at (0,0) [shape=circle,draw,minimum size=6mm] {$1$};
\node (d2) at (3,0)  [shape=circle,draw,minimum size=6mm] {$2$};
\node (d3) at (6,0){};

\node (dots) at (6.5,0){$\cdots$};

\node (d4) at (7,0){};
\node (dd) at (10,0)  [shape=circle,draw,minimum size=6mm] {$d$};

\draw[-{triangle 45}] (d1) to [in=75, out=105, loop above] node {$(1,k_1)$} ();
\draw[-{triangle 45}] (d1) to [out=20, in=160] node {$*_1$} (d2);

\draw[-{triangle 45}] (d2) to [out=200, in=340] node {$(2,k_2)$} (d1);
\draw[-{triangle 45}] (d2) to [out=20, in=160] node {$*_2$} (d3);

\draw[-{triangle 45}] (d4) to [out=20, in=160] node {$*_{d-1}$} (dd);

\draw[-{triangle 45}] (dd) to [out=220, in=320] node {$(d,k_d)$} (d1);
\end{tikzpicture}
\caption{The choice of labels for the graph of $X_C$.}
\label{SFTconj1}
\end{figure}

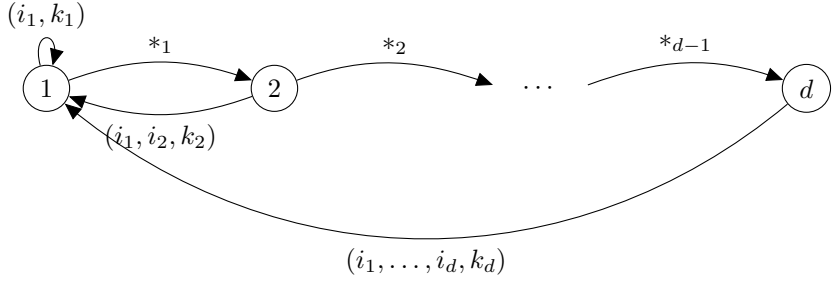
\begin{figure}
\centering
\begin{tikzpicture}[auto]
\node (d1) at (0,0) [shape=circle,draw,minimum size=6mm] {$1$};
\node (d2) at (3,0)  [shape=circle,draw,minimum size=6mm] {$2$};
\node (d3) at (6,0){};

\node (dots) at (6.5,0){$\cdots$};

\node (d4) at (7,0){};
\node (dd) at (10,0)  [shape=circle,draw,minimum size=6mm] {$d$};

\draw[-{triangle 45}] (d1) to [in=75, out=105, loop above] node {$(i_1,k_1)$} ();
\draw[-{triangle 45}] (d1) to [out=20, in=160] node {$*_1$} (d2);

\draw[-{triangle 45}] (d2) to [out=200, in=340] node {$(i_1,i_2,k_2)$} (d1);
\draw[-{triangle 45}] (d2) to [out=20, in=160] node {$*_2$} (d3);

\draw[-{triangle 45}] (d4) to [out=20, in=160] node {$*_{d-1}$} (dd);

\draw[-{triangle 45}] (dd) to [out=220, in=320] node {$(i_1,\dots,i_d,k_d)$} (d1);
\end{tikzpicture}
\caption{The choice of labels for the graph of $X_B$.}
\label{SFTconj2}
\end{figure}

The labeling has been chosen in a way that suggests the correct choice of the reversible sliding block code $\phi:S_n\times X_C\to X_B$. For any $x\in S_n\times X_C$ we make the usual identification $x=(x_1,x_2)$ where $x_1\in S_n$, $x_2\in X_C$ and we denote $\pi_1(x)=x_1$, $\pi_2(x)=x_2$. Then $\phi$ is defined by

\begin{flalign*}
\phi(x)[i]=
\left\{
\begin{array}{l}
*_j\text{ when } \pi_2(x)[i]=*_j, \\
(i_1,k_1)\text{ when } \pi_2(x)[i]=(1,k_1) \mbox{ and } \pi_1(x)[i]=i_1, \\
(i_1,i_2,\dots,i_j,k_j)\mbox{ when } \pi_2(x)[i-(j-1),i]=*_1*_2\cdots *_{j-1}(j,k_j) \\ \qquad \mbox{ and }\pi_1(x)[i-(j-1),i]=i_1 i_2\cdots i_j
\end{array}
\right.
\end{flalign*}
The intuition here is that the sliding block code $\phi$ attempts to deposit all the information at $x[i]$ to $\phi(x)[i]$. This is not possible when $\pi_2(x)[x]=*_j$, so the remaining information is deposited to the nearest suitable coordinate to the right.

For the converse, assume that the word $(na_{d-1})\cdots (n^d a_0)$ is not lexicographically greater than all its shifts and that $S_{n\gamma}$ is an SFT. Then $n\gamma$ is the unique positive solution of some equation $x^e=b_{e-1}x^{e-1}+\dots+b_0$ where $e\geq 1$, $b_{e-1},b_0\geq 1$ and $b_i\geq 0$ such that $b_{e-1}\cdots b_{0}$ is lexicographically greater than all its shifts. As above, $S_{n\gamma}$ is conjugate to an edge shift $Y$ given by a matrix with eigenvalues $\beta_1,\beta_2,\dots,\beta_e$ which are also the roots of the corresponding polynomial. By our assumption the polynomials $x^e-b_{e-1}x^{e-1}-\dots-b_0$ and $y^d-na_{d-1}y^{d-1}-\cdots-n^d a_0$ are different, so they have different sets of roots (with multiplicities taken into account) and
\[\zeta_Y(t)=\prod_j(1-\beta_i t)^{-1}\neq \prod_i(1-n\gamma_i t)^{-1}=\zeta_X(t),\]
because $\C[t]$ is a unique factorization domain.
\end{proof}

We conclude with an example concerning an SFT beta-shift $S_{\beta_1\times\beta_2}$ where the assumption of either $\beta_1$ or $\beta_2$ being an integer is dropped.

\begin{example}
A beta-shift $S_{\gamma\times\gamma}$ can be topologically direct prime even if $S_{\gamma}$ and $S_{\gamma\times\gamma}$ are SFTs (and then in particular $S_{\gamma\times\gamma}$ is not conjugate to $S_{\gamma}\times S_{\gamma}$). Denote by $\gamma$ the unique positive root of $x^3-x^2-x-1$. By Lemma~\ref{bSFTlemma} we have $d(\gamma)=111$ and in particular $S_\gamma$ is an SFT. Denote $\beta=\gamma^2$. Its minimal polynomial is $x^3-3x^2-x-1$ and by Lemma~\ref{bSFTlemma} $d(\beta)=311$, so $S_\beta$ is an SFT and it is conjugate to the edge SFT given by the matrix $A=\left(\begin{smallmatrix} 3 & 1 & 0 \\ 1 & 0 & 1 \\ 1 & 0 & 0 \end{smallmatrix}\right)$. It has three distinct eigenvalues $\beta_0=\beta$, $\beta_1$ and $\beta_2$.

We claim that $S_\beta$ is topologically direct prime. To see this, assume to the contrary that $S_\beta$ is topologically conjugate to $X\times Y$ where $X$ and $Y$ are nontrivial direct factors for $S_\beta$. Since $X\times Y$ is a mixing SFT, it follows from Proposition~6 of~\cite{Lind84} that $X$ and $Y$ are mixing SFTs and in particular they are infinite. The zeta functions of $X$ and $Y$ are of the form
\[\zeta_X(t)=\prod_i(1-\mu_i t)^{-1} \quad \mbox{and} \quad \zeta_Y(t)=\prod_j(1-\nu_j t)^{-1}\]
for some $\mu_i,\nu_j\in\C\setminus\{0\}$. The zeta-function of $S_\beta$ is
\[\zeta_{S_\beta}(t)=(1-\beta t)^{-1}(1-\beta_1 t)^{-1}(1-\beta_2 t)^{-1}=\prod_{i,j}(1-\mu_i\nu_j t)^{-1}.\]
Because $\C[t]$ is a unique factorization domain and because $X$ and $Y$ are non-trivial SFTs, we may assume without loss of generality that $\zeta_X(t)=(1-\mu t)$ and $\zeta_Y(t)=(1-\nu_1 t)(1-\nu_2 t)(1-\nu_3 t)$ for some $\mu,\nu_1,\nu_2,\nu_3\in\C\setminus\{0\}$. The quantities $\mu$ and $\nu_1+\nu_2+\nu_3$ are the numbers of $1$-periodic points of $X$ and $Y$ respectively and thus the number of $1$-periodic points of $S_\beta$ is equal to $\mu(\nu_1+\nu_2+\nu_3)=3$ where $\mu$ and $\nu_1+\nu_2+\nu_3$ are nonnegative integers. In particular $\mu\in\{1,3\}$.

Assume first that $\mu=1$. Therefore $X$ has the same zeta function as the full shift over the one letter alphabet and $X$ has just one periodic point. As a mixing SFT $X$ has periodic points dense so $X$ only contains one point, contradicting the nontriviality of $X$.

Assume then that $\mu=3$. Therefore $X$ has the same zeta function as $\digs_3^\Z$ and $X$ has precisely $3^n$ $n$-periodic points for all $n\in\Npos$. In particular the number of $2$-periodic points of $S_\beta$ is divisible by $3^2=9$. On the other hand the number of $2$-periodic points of $S_\beta$ is equal to $\tr(A^2)=11$, a contradiction.

%We also note that the topological direct primeness $S_\beta$ cannot be shown purely by an entropy argument as in the paper of Lind, because $S_\beta$ and $S_\gamma\times S_\gamma$ have equal entropies.
\end{example}

\bibliographystyle{plain}
\bibliography{mybib}{}

\end{document}